\documentclass[a4paper,12pt]{amsart}

\usepackage{amsmath}
\usepackage{amssymb}
\usepackage{mathrsfs}
\usepackage{enumerate}
\usepackage{ifthen}
\usepackage{graphicx, color}
\usepackage[T1]{fontenc} %skandit

\nonstopmode \numberwithin{equation}{section}
\newtheorem{thm}{Theorem}[section]
\newtheorem{cor}{Corollary}[section]
\newtheorem{lem}{Lemma}[section]

\theoremstyle{definition}

\newtheorem{example}{Example}[section]

\newcounter{minutes}
\divide\time by 60
\newcounter{hours}
\multiply\time by 60
\addtocounter{minutes}{-\time}
\newcounter {own}
\def\theown {\thesection       .\arabic{own}}

{\qed\bigskip}

\newcounter{alphabet}

\begin{document}

\title{PRE-SCHWARZIAN AND SCHWARZIAN NORM ESTIMATES FOR HARMONIC FUNCTIONS WITH FIXED ANALYTIC PART}

 \author{Md Firoz Ali}
\address{Md Firoz Ali,
Department of Mathematics,
National Institute of Technology Durgapur,
West Bengal-713209, India.}
\email{ali.firoz89@gmail.com, fali.maths@nitdgp.ac.in}

\author{Sushil Pandit}
\address{Sushil Pandit,
Department of Mathematics,
National Institute of Technology Durgapur,
West Bengal-713209, India}
\email{sushilpandit15594@gmail.com}


\def\thefootnote{}
\footnotetext{ {\tiny File:~\jobname.tex,
printed: \number\year-\number\month-\number\day,
          \thehours.\ifnum\theminutes<10{0}\fi\theminutes }
} \makeatletter\def\thefootnote{\@arabic\c@footnote}\makeatother

\begin{abstract}
In the present article, we discuss about the estimate of the pre-Schwarzian and Schwarzian norms for locally univalent harmonic functions $f=h+\overline{g}$ in the unit disk $\mathbb{D}:=\{z\in\mathbb{C}:\, |z|<1\}$. First, we prove a general result for the estimate of the pre-Schwarzian norm which rectify few earlier flawed results. We also consider a new class $\mathcal{F}_0$ consisting of all harmonic functions $f=h+\overline{g}$ in the unit disk $\mathbb{D}$ such that ${\rm Re\,}\left(1+z\frac{h''(z)}{h'(z)}\right)>0$ for $z\in\mathbb{D}$  with dilatation $\omega_f(z)\in Aut(\mathbb{D})$ and obtain best possible estimates of the pre-Schwarzian and Schwarzian norms for functions in the class  $\mathcal{F}_0$. Moreover, we obtain the distortion and coefficient estimates of the co-analytic function $g$ when $f=h+\overline{g}\in\mathcal{F}_0$.
\end{abstract}

\thanks{}

\maketitle
\pagestyle{myheadings}
\markboth{Md Firoz Ali, Sushil Pandit}{Pre-Schwarzian and Schwarzian norm estimates}

\section{Introduction}
A twice continuously differentiable complex valued function $f=u+iv$ in a domain $\Omega$ is called harmonic if $u$ and $v$ both are harmonic in $\Omega$ or equivalently if it satisfies the Laplace equation $\Delta f = 4f_{z\overline{z}} = 0.$ In a simply connected domain $\Omega,$ every harmonic mapping $f$ has a canonical representation of the form $f = h+\overline{g},$ where $h$ and $g$ are analytic functions in  $\Omega$ called the analytic and co-analytic part of $f$ respectively. The Jacobian of $f=h+\overline{g}$ is defined by
$ J_f(z)=|f_z|^2 - |f_{\overline{z}}|^2=|h'(z)|^2-|g'(z)|^2.$
%\hspace{.14cm}
%
%\noindent
The harmonic mapping $f$ is called orientation preserving or sense preserving mapping if $J_f(z)>0$ and is called orientation reversing or sense reversing mapping if $J_f(z)<0$. For a sense preserving harmonic mapping $f=h+\overline{g}$, the dilatation $\omega_f = g'/h'$ has the property that $|\omega_f(z)|<1$ in $\Omega.$ According to Lewy's theorem \cite{Lewy-1936}, a harmonic mapping $f = h+\overline{g}$ is locally univalent in a domain  $\Omega$ if its Jacobian $J_f\ne 0$. \\

Let $\mathcal{A}$ denote the class of all analytic functions $h$ normalized by $h(0)=h'(0)-1=0$ in the unit disk $\mathbb{D}:=\{z\in\mathbb{C}:|z|<1\}$ and $\mathcal{S}$ denote the class of all univalent functions in $\mathcal{A}$. Let $\mathcal{H}$ denote the class of all harmonic mappings $f = h+\overline{g}$ in the unit disk $\mathbb{D}$ with the normalization $h(0) = h'(0)-1 = g(0) = 0$. Hence, a function $f=h+\overline{g}$ in $\mathcal{H}$ has the the form
\begin{align}\label{p1-001}
h(z)=z+\sum\limits_{n=2}^\infty a_nz^n ~~\quad\text{and}~~\quad g(z)=\sum\limits_{n=1}^\infty b_nz^n.
\end{align}
Let $\mathcal{S_H}$ be the subclass of $\mathcal{H}$ of all sense preserving and univalent harmonic mappings. The class $\mathcal{S_H}$ is not compact but the class $\mathcal{S}^{0}_\mathcal{H} = \{f\in \mathcal{S_H} : g'(0)=0\}$ is compact. A univalent harmonic mapping $f = h+\overline{g}$ is called convex if the image $f(\mathbb{D})$ is a convex domain. Let $\mathcal{K_H}$ denote the subclass of $\mathcal{S_H}$ of all convex harmonic mappings and $\mathcal{K}^{0}_\mathcal{H} = \{f\in \mathcal{K_H} : g'(0)=0\}.$ A domain $\Omega$ is close-to-convex if $\mathbb{C}\setminus\Omega$ can be represented as union of non-intersecting half-lines. A harmonic mapping $f = h+\overline{g}$ is called close-to-convex in $\mathbb{D}$ if the image $f(\mathbb{D})$ is a close-to-convex domain. Suppose $\mathcal{C_H}$ denote the subclass of $\mathcal{S_H}$ of all close-to-convex harmonic mappings in the unit disk $\mathbb{D}$ and $\mathcal{C}^{0}_\mathcal{H}= \{f=h+\overline{g}\in \mathcal{C_H} : g'(0)=0\}$. For more interesting facts and results on planar harmonic univalent mappings, we refer to \cite{Duren-2004}.

\subsection{Pre-Schwarzian and Schwarzian norm}
For a locally univalent analytic function $f$ defined in a simply connected domain $\Omega$, the pre-Schwarzian derivative $P_f$ and the Schwarzian derivative $S_f$ are defined as
\begin{align}\label{p1-005}
P_f(z)=\frac{f''(z)}{f'(z)}\quad\text{and}\quad S_f(z) = P_f'(z)-\frac{1}{2}P_f^2(z)=\frac{f'''(z)}{f''(z)}-\frac{3}{2}\left(\frac{f''(z)}{f'(z)}\right)^2
\end{align}
respectively. Moreover, the pre-Schwarzian norm and the Schwarzian norm of $f$ are defined by
\begin{align}\label{p1-010}
||P_f|| = \sup_{z \in \mathbb{D}}(1-|z|^2)|P_f(z)|\quad\text{and}\quad ||S_f|| = \sup_{z \in \mathbb{D}}(1-|z|^2)^2|S_f(z)|
\end{align}
respectively. Several important global univalence criteria for a locally univalent analytic function $f$ were obtained using the notions of pre-Schwarzian and Schwarzian derivatives of $f$. For a univalent function $f$, it is well known that $||P_f||\leq 6$ and $||S_f||\leq 6$ (see \cite{Kruas-1932}) and these estimates are best possible. On the other hand, for a locally univalent function $f$ in $\mathcal{A}$, it is also known that if $||P_f||\leq 1$ (see \cite{Becker-1972}, \cite{Becker-Pommerenke-1984}) or $||S_f||\leq 2$ (see \cite{Nehari-1949}), then the function $f$ is univalent in $\mathbb{D}$. In 1976, Yamashita \cite{Yamashita-1976} proved that $||P_f ||$ is finite if and only if $f$ is uniformly locally univalent in $\mathbb{D},$ that is, there exists a constant $\rho>0$ such that $f$ is univalent on the hyperbolic disk $|(z-a)/(1-\overline{a}z)|<\tanh\rho$ of radius $\rho$ for every $a\in\mathbb{D}.$\\

In 2003, Chuaqui et al. \cite{Chuaqui-Duren-Osgood-2003} defined Schwarzian derivative (which we denote by $\mathbb{S}_f$) for a locally univalent and sense preserving harmonic mapping $f = h+\overline{g}$ with its dilatation $\omega = g'/h'$ of the form $\omega = q^2$ for some analytic function $q$, by the formula
\begin{align}\label{p1-015}
\mathbb{S}_f
&= 2[(\log\lambda)_{zz}-((\log\lambda)_z)^2],~~\text{where}~\lambda = |h'|+|g'|\\
&= S_h+\frac{2\overline{q}}{1+|q|^2}\left(q''-\frac{h''}{h'}q'\right)-4\left(\frac{q'\overline{q}}{1+|q|^2}\right)^2.\nonumber
\end{align}
The pre-Schwarzian derivative (which we denote by $\mathbb{P}_f$) for a locally univalent and sense preserving harmonic mapping $f = h+\overline{g}$ with its dilatation $\omega = g'/h'$ of the form $\omega = q^2$ for some analytic function $q$, first proposed by Kanas and Klimek-Sm\c{e}t \cite{Kanas-Smet-2014}, is defined as
\begin{align}\label{p1-020}
\mathbb{P}_f = \frac{2\partial(\log\lambda)}{\partial z}= \frac{h''}{h'}+\frac{2\overline{q}q'}{1+|q|^2},~~\text{where}~\lambda = |h'|+|g'|.
\end{align}
As in the case of analytic functions, the pre-Schwarzian norm $||\mathbb{P}_f||$ and the Schwarzian norm $||\mathbb{S}_f||$ of a sense-preserving locally univalent harmonic mapping $f = h+\overline{g}$ with dilatation square of some analytic function in the unit disk $\mathbb{D}$ are defined by \eqref{p1-010}.
The Schwarzian and pre-Schwarzian derivatives defined in \eqref{p1-015} and \eqref{p1-020} respectively have one disadvantage that comes from the fact that the dilatation has restriction. Due to this in many cases one can not define it globally  for an univalent harmonic mapping.
In 2015, Hern{\'a}ndez and Mart{\'\i}n \cite{Hernandez-Martin-2015} defined the Schwarzian derivative of a locally univalent harmonic mapping $f=h+\overline{g}$ by\\
\begin{align}\label{p1-025}
S_f &= \left(\log J_f\right)_{zz}-\frac{1}{2}\left(\log J_f\right)_z^2\\
&= S_h+\frac{\overline{\omega}}{1-|\omega|^2}\left(\frac{h''}{h'}\omega'-\omega''\right)-\frac{3}{2}\left(\frac{\omega'\overline{\omega}}{1-|\omega|^2}\right)^2,\nonumber
\end{align}
where $J_f$ is the Jacobian of $f$ and $S_h$ is the classical Schwarzian derivative of the analytic function $h$ and $\omega$ is the dilatation of $f.$ The pre-Schwarzian derivative of $f=h+\overline{g}$ is defined as
\begin{align}\label{p1-027}
P_f = \left(\log J_f\right)_z = \frac{h''}{h'}-\frac{\overline{\omega}\omega'}{1-|\omega|^2}.
\end{align}
This is a generalization of the classical pre-Schwarzian derivative and Schwarzian derivative of an analytic function. Note that when $f$ is analytic, the dilatation $\omega=0.$ It is also easy to see that $S_f = (P_f)_z-\frac{1}{2}(P_f)^2.$ The pre-Schwarzian and Schwarzian derivatives of harmonic function have the chain rule property exactly in the same form as in the analytic case. If $f$ is a sense preserving harmonic function and $\varphi$ is a locally univalent analytic function for which the composition $f\circ \varphi$ is defined, then
\begin{align*}
P_{f\circ \varphi}(z) &= P_f\circ \varphi(z)\cdot \varphi'(z)+P_{\varphi}(z),\\
S_{f\circ \varphi}(z) &= S_f\circ \varphi(z)\cdot (\varphi'(z))^2+S_{\varphi}(z).
\end{align*}
Both pre-Schwarzian and Schwarzian derivatives are invariant under affine transformation of a harmonic function $f$ i.e., if $A(w) = aw +b\overline{w} +c$, $|a| > |b|$, then
$$
P_{A\circ f} \equiv P_f ~~\text{and}~~ S_{A\circ f} \equiv S_f.
$$
Again, as in the case of analytic functions, the pre-Schwarzian norm $||P_f||$ and the Schwarzian norm $||S_f||$ of a sense-preserving locally univalent harmonic mapping $f = h+\overline{g}$ in the unit disk $\mathbb{D}$ are defined by \eqref{p1-010}. Hern{\'a}ndez and Mart{\'\i}n \cite{Hernandez-Martin-2015} proved that a sense-preserving harmonic mappings is uniformly locally univalent if and only if its Schwarzian norm is finite.
Later, Liu and Ponnusamy \cite{Liu-Ponnusamy-2018} proved that a sense-preserving harmonic mapping is uniformly locally univalent if and only if its pre-Schwarzian norm is finite.

%\section{A note on Pre-Schwarzian norm}
\section{Main Results}
Analytic parts of harmonic mappings play an important role to model their geometric properties. For instance, if $f=h+\overline{g}\in\mathcal{H}$ is a sense preserving harmonic mapping and $h$ is analytic convex univalent, then $f\in\mathcal{S_H}$ and maps $\mathbb{D}$ onto a close-to-convex domain \cite{Clunie-Sheil-1984}. In 2011, Bshouty and Lyzzaik \cite{Bshouty-Lyzzaik-2011} proved that if $f=h+\overline{g}\in\mathcal{H}$ is a harmonic mapping with dilatation $\omega=z$ and the analytic part $h \in \mathcal{A}$ satisfies ${\rm Re\,}Q_h(z)>-\frac{1}{2}$ for $z\in\mathbb{D}$, where
\begin{equation}\label{p1-028}
Q_h(z):=1+\frac{zh''(z)}{h'(z)}
\end{equation}
then $f$ is an univalent close-to-convex function. Throughout this article, we will use the notation $Q_h(z)$ defined in (\ref{p1-028}) for a function $h$ in $\mathcal{A}$.\\

\subsection{The pre-Schwarzian norm $||\mathbb{P}_f||$}
Estimate of the pre-Schwarzian norm $||\mathbb{P}_f||$ of sense-preserving harmonic mappings of the form $f=h+\overline{g}$ where $h$ belongs to some subclasses of $\mathcal{S}$ and the dilatation $\omega$ is a square of an analytic self mapping of unit disk have been studied by some authors (see \cite{Kanas-Maharana-Prajapat-2019, Kanas-Smet-2014, Prajapat-Manivannan-Maharana-2020, Rajbala-Prajapat-2021}).
But we observed some flaw in these articles which we will discuss in our concluding remarks (see Section \ref{Concluding Remarks}). In this direction, we prove a general result to estimate the pre-Schwarzian norm $||\mathbb{P}_f||$ of a sense-preserving harmonic mapping $f=h+\overline{g}$ with dilatation $\omega=q^2$ where $q$ is an analytic self mapping of the unit disk.

\begin{thm}\label{p1-029}
Let $f=h+\overline{g}\in\mathcal{H}$ be a sense-preserving harmonic mapping with dilatation $\omega(z)=q^2(z)=g'(z)/h'(z)$ where $q:\mathbb{D}\rightarrow\mathbb{D}$ is an analytic function. Then either $||\mathbb{P}_f||=||P_h||=\infty$ or both $||\mathbb{P}_f||$ and $||P_h||$ are finite. In the later case,
\begin{align*}
\Big|||\mathbb{P}_f||-||P_h||\Big|\leq \frac{2r_0(1-r_0^2)}{1+r_0^2}\approx 0.6005...,
\end{align*}
where $r_0=\sqrt{\sqrt{5}-2}$ and the estimate is sharp.
\end{thm}

A similar result for the pre-Schwarzian norm $||P_f||$ has been proved by Liu and Ponnusamy \cite{Liu-Ponnusamy-2018}. Proof of the Theorem \ref{p1-029} will be provided in Section \ref{Proof of Main Results}. Using this general result, one can obtain estimate of the pre-Schwarzian norm $||\mathbb{P}_f||$ of sense-preserving harmonic mappings of the form $f=h+\overline{g}$ where $h$ belongs to some subclasses of $\mathcal{S}$ and the dilatation $\omega$ is a square of an analytic self mapping of unit disk. Here we mention three such applications which plays important role to rectify few flawed results in harmonic function theory (see Section \ref{Concluding Remarks}).

\begin{cor}\label{p1-030}
Let $f=h+\overline{g}\in\mathcal{H}$ be a sense-preserving harmonic mapping with dilatation $\omega(z)=q^2(z)=g'(z)/h'(z)$ where $q:\mathbb{D}\rightarrow\mathbb{D}$ is an analytic function and $h$ is such that ${\rm Re\,}Q_h(z)>-\frac{1}{2}$ for $z\in\mathbb{D}.$ Then
$$||\mathbb{P}_f||\leq 6+\frac{2r_0(1-r_0^2)}{1+r_0^2}\approx 6.6005\ldots,~~\text{where}~~r_0=\sqrt{\sqrt{5}-2}.$$
\end{cor}

\begin{cor}\label{p1-031}
Let $f=h+\overline{g}\in\mathcal{H}$ be a sense-preserving harmonic mapping with dilatation $\omega(z)=q^2(z)=g'(z)/h'(z)$ where $q:\mathbb{D}\rightarrow\mathbb{D}$ is an analytic function and $h$ is such that ${\rm Re\,}Q_h(z)<\frac{3}{2}$ for $z\in\mathbb{D}.$ Then
$$||\mathbb{P}_f||\leq 2+\frac{2r_0(1-r_0^2)}{1+r_0^2}\approx 2.6005\ldots,~~\text{where}~~r_0=\sqrt{\sqrt{5}-2}.$$
\end{cor}

\begin{cor}\label{p1-033}
Let $f=h+\overline{g}\in\mathcal{H}$ be a sense-preserving harmonic mapping with dilatation $\omega(z)=q^2(z)=g'(z)/h'(z)$ where $q:\mathbb{D}\rightarrow\mathbb{D}$ is an analytic function and $h$ is such that $-\frac{1}{2}<{\rm Re\,}Q_h(z)<\frac{3}{2}$ for $z\in\mathbb{D}.$ Then
$$||\mathbb{P}_f||\leq 2+\frac{2r_0(1-r_0^2)}{1+r_0^2}\approx 2.6005\ldots,~~\text{where}~~r_0=\sqrt{\sqrt{5}-2}.$$
\end{cor}

\subsection{The class $\mathcal{F}_0 $}
In this article, we also introduce a new class $\mathcal{F}_0$ of harmonic mappings $f=h+\overline{g}$ in $\mathcal{H}$ such that
$${\rm Re\,}\left(1+\frac{zh''(z)}{h'(z)}\right)>0~\quad\text{for}\quad z\in\mathbb{D}$$
and dilatation $\omega_f=\frac{g'}{h'}\in Aut(\mathbb{D})$, where $Aut(\mathbb{D})$ denote the collection of all automorphisms of $\mathbb{D}$ i.e.,
\begin{align*}
Aut(\mathbb{D})= \left\{e^{i\theta}\frac{z-a}{1-\overline{a}z}:a\in\mathbb{D}~,~\theta\in\mathbb{R}\right\}.
\end{align*}
Clearly, if  $f=h+\overline{g}\in\mathcal{F}_0$ then $h$ is a convex univalent function of the form \eqref{p1-001}.
Thus every $f\in\mathcal{F}_0$ maps the unit disk $\mathbb{D}$ onto a close-to-convex domain \cite{Clunie-Sheil-1984}. As the dilatation $\omega_f\in Aut(\mathbb{D})$, it is of the form
\begin{align*}
\omega_f(z)= e^{i\theta}\frac{z+ \alpha}{1+\overline{\alpha}z}~~\text{for some}~~ \theta\in\mathbb{R}, ~~\alpha\in\mathbb{D}.
\end{align*}
Therefore,
\begin{align*}
\omega_f'(z)= e^{i\theta}\frac{1-|\alpha|^2}{(1+\overline{\alpha}z)^2}\quad\text{and}\quad \omega_f''(z)= -2\overline{\alpha}\frac{(1-|\alpha|^2)}{(1+\overline{\alpha}z)^3}e^{i\theta},
\end{align*}
and so,
\begin{equation*}\label{p1-035}
\omega_f'(0)= e^{i\theta}(1-|\alpha|^2)\quad\text{and}\quad  \omega_f''(0)= -2\overline{\alpha}(1-|\alpha|^2)e^{i\theta}.
\end{equation*}

A family $\mathcal{F}$ of sense-preserving harmonic mappings $f = h + \overline{g}$ in $\mathbb{D},$ normalized by $ h(0) = g(0) = 0$ and $h'(0) = 1$ is said to be linear invariant if it is closed under the  Koebe transform
\begin{align*}
   L_\phi(f)(z) = \frac{f(\phi(z))-f(\phi(0))}{f_z(\phi(0))\phi'(0)},\quad \phi\in Aut(\mathbb{D}),
\end{align*}
and is said to be affine invariant if it is closed under the affine transform
\begin{align*}
A_{\epsilon}(f)(z) = \frac{f(z)-\overline{\epsilon f(z)}}{1-\overline{\epsilon}g'(0)}, \quad |\epsilon|<1.
\end{align*}
Few special examples of affine and linear invariant families are the classes $\mathcal{S_H}$, $\mathcal{K_H}$ and $\mathcal{C_H}$ of sense-preserving univalent, convex and close-to-convex harmonic mappings, respectively in the unit disk. Other examples of affine and linear invariant families of sense-preserving harmonic mappings are the
stable harmonic univalent ($\mathcal{SHU}$) and the stable harmonic convex ($\mathcal{SHC}$) classes. A
function $f = h + \overline{g}\in\mathcal{S_H}$ is $\mathcal{SHU}$ (resp. $\mathcal{SHC}$) if $h +\lambda \overline{g}$ is univalent (convex)
for every $|\lambda| = 1$. These classes are affine and linear invariant because univalence and
convexity are preserved under the Koebe transform and affine transform.\\

The class $\mathcal{F}_0$ is linear invariant but not affine invariant. For, if $\phi\in Aut(\mathbb{D})$, then
\begin{align}\label{p1-040}
F(z):=L_\phi(f)(z) = \frac{f(\phi(z))-f(\phi(0))}{f_z(\phi(0))\phi'(0)}= H(z)+\overline{G(z)},
\end{align}
where
$$H(z)= \frac{h(\phi(z))-h(\phi(0))}{f_z(\phi(0))\phi'(0)}\quad \text{and}\quad G(z)=\overline{\frac{g(\phi(z))-g(\phi(0))}{\overline{f_z(\phi(0))\phi'(0)}}}.$$
 Since $h\circ\phi$ is a convex univalent function in $\mathbb{D},$ it follows that
\begin{equation}\label{p1-045}
{\rm Re\,}\left(1+\frac{zH''(z)}{H'(z)}\right)= {\rm Re\,}\left(1+\frac{z(h\circ\phi)''(z)}{(h\circ\phi)'(z)}\right)>0~~\text{for}~~z\in\mathbb{D}.
\end{equation}
Let $\omega_F(z)$ be the dilatation of $F(z)=L_\phi(f)(z)$. Then
\begin{equation}\label{p1-050}
\begin{split}
\omega_F(z)=\frac{G'(z)}{H'(z)}=\frac{(g\circ\phi)'(z)}{(h\circ\phi)'(z)}\frac{f_z(\phi(0))}{\overline{f_z(\phi(0))}}
 = e^{i\theta_1}(\omega_f\circ\phi)(z)
 = e^{i\theta_2}\frac{z+\alpha}{1+\overline{\alpha}z}
\end{split}
\end{equation}
for some $\alpha\in\mathbb{D}$ and $\theta_1,~~ \theta_2\in\mathbb{R}$. By \eqref{p1-045} and \eqref{p1-050}, we see that $F=L_\phi(f)\in\mathcal{F}_0.$
But, if $f=h+\overline{g}\in\mathcal{F}_0$ with $h(z)=z$ and dilatation $\omega_f(z)=z$, then a simple calculation shows that the affine transformation $$A_\gamma(f)(z)=[f(z)+\gamma\overline{f(z)}]/[1+\gamma g'(0)]$$ for $|\gamma|<1$ is not a member of $\mathcal{F}_0$, which ensure that the family $\mathcal{F}_0$ is not affine invariant.\\

%\subsection{New pre-Schwarzian norm $||P_f||$ and Schwarzian norm $||S_f||$}
For a convex harmonic mapping $f=h+\overline{g}$, Hern{\'a}ndez and Mart{\'\i}n \cite{Hernandez-Martin-2015} proved that $\|P_f\|\le 5$ and $\|S_f\|\le 6$. Although the first estimate is sharp but it is unknown whether the later estimate is sharp or not. In 2016, Graf \cite{GRAF-2016} proved that
\begin{align}\label{p1-055}
|P_f(z)|\leq~ \frac{2(\alpha_0+|z|)}{1-|z|^2},
\end{align}
for any locally univalent harmonic mapping $f$ in a affine and linear invariant family $\mathcal{F}$ where $\alpha_0 = \sup_{f \in \mathcal{F}^0}|a_2|$ with $\mathcal{F}^0 = \{ f \in \mathcal{F}: g'(0)=0\}.$ The estimate \eqref{p1-055} is sharp for example in affine and linear invariant families $\mathcal{K}_H$ and $\mathcal{C}_H$ of univalent convex and close-to-convex harmonic functions respectively. Estimates of the Schwarzian norm $\|S_f\|$ for stable harmonic univalent ($\mathcal{SHU}$) and the stable harmonic convex ($\mathcal{SHC}$) functions (which are affine and linear invariant families) are obtained by Chuaqui et al. \cite{Chuaqui-Hernandez-Martin-2017}. On the other hand, sharp estimates of the pre-Schwarzian norm $\|P_f\|$ for $\mathcal{SHU}$ and $\mathcal{SHC}$ functions are obtained by Liu and Ponnusamy \cite{Liu-Ponnusamy-2018}.\\

Motivated by these facts, we consider the linear invariant family $\mathcal{F}_0$ and obtain best possible estimate of pre-Schwarzian and Schwarzian norms for the functions in the class $\mathcal{F}_0$.

\begin{thm}\label{p1-057}
If $f=h+\overline{g}\in \mathcal{F}_0$ is of the form \eqref{p1-001} then $||P_f||\leq5$. The estimate is best possible.
\end{thm}

\begin{thm}\label{p1-060}
If $f=h+\overline{g}\in \mathcal{F}_0$ is of the form \eqref{p1-001} then $||S_f||\leq3$. The estimate is best possible.
\end{thm}

The classical Bloch theorem asserts the existence of a positive constant $m$ such that for any holomorphic mapping $f$ in the unit disk $\mathbb{D}$, with the normalization $f'(0) =1,$ the image $f(\mathbb{D})$ contains a Schlicht disk of radius $m$. By Schlicht disk, we mean a disk which is the univalent image of some region in $\mathbb{D}$. The Bloch constant is defined as the supremum of such constants $m.$ A harmonic mapping $f=h+\overline{g}\in\mathcal{H}$ is called Bloch mapping if and only if
\begin{align*}
\mathcal{B}_f=\sup_{z\in\mathbb{D}}(1-|z|^2)(|h'(z)|+|g'(z)|)<\infty
\end{align*}
where $\mathcal{B}_f$ is called Bloch constant of $f.$ Chen et al. \cite{Chen-Gauthier-Hengartner-2000} estimated Bloch constant for harmonic mappings. For more information about Bloch constant, we refer \cite{Chen-Gauthier-Hengartner-2000,Kanas-Maharana-Prajapat-2019,Colonna-1989}. Now we will see that there exists some functions $f=h+\overline{g}\in\mathcal{F}_0$ for which the Bloch constant $\mathcal{B}_f$ is unbounded, that is, functions in $\mathcal{F}_0$ are not necessarily Bloch.

\begin{example}
Let $f=h+\overline{g}\in\mathcal{F}_0$ with $\displaystyle h(z)=\frac{z}{1-e^{i\theta}z},~~\theta\in\mathbb{R},$ and the dilatation $\omega(z)=e^{i\alpha}z,~~\alpha\in\mathbb{R}.$ Then the Bloch constant $\mathcal{B}_f$ of $f$ is
\begin{align*}
\mathcal{B}_f=& \sup_{z\in\mathbb{D}}(1-|z|^2)(|h'(z)|+|g'(z)|)\\
= & \sup_{z\in\mathbb{D}}(1-|z|^2)(1+|\omega(z)|)|h'(z)|\\
= & \sup_{z\in\mathbb{D}}(1-|z|^2)(1+|z|)\frac{1}{|1-e^{i\alpha }z|^2}.
\end{align*}
When $z=e^{-i\alpha}r\in \mathbb{D}$, $0\leq r<1$, we have
\begin{align*}
(1-|z|^2)(1+|z|)\frac{1}{|1-e^{i\alpha }z|^2} =\frac{(1-r^2)(1+r)}{(1-r)^2},
\end{align*}
which tends to $\infty$ as $r\to 1$ and so $\mathcal{B}_f$ is unbounded. Hence, the function $f$ is not Bloch.
\end{example}

\begin{example}
Let $f=h+\overline{g}\in\mathcal{F}_0$ with $h(z)$ be such that $h'(z)=\frac{1}{1-z}$ and dilatation $\omega(z)=z.$ It is easy to see that
\begin{align*}
\mathcal{B}_f=\sup_{z\in\mathbb{D}}(1-|z|^2)(1+|\omega(z)|)|h'(z)| \leq \sup_{z\in\mathbb{D}}(1+|z|)^2=4.
\end{align*}
So the function $f$ is Bloch.
\end{example}

If $f=h+\overline{g}\in\mathcal{F}_0$ then $h\in\mathcal{A}$ is a convex univalent function for which the estimate of the Taylor coefficient and distortion theorem are well-known in the literature (see \cite{Goodman-1983}).
Our next two results give coefficient estimates and distortion theorem for the co-analytic part $g$ for functions in the class $\mathcal{F}_0$.

\begin{thm}\label{p1-065}
Let $f=h+\overline{g}\in\mathcal{F}_0$ is of the form \eqref{p1-001}
then
\begin{align*}
|b_n|\leq 1,~~\text{for all}~~n\geq1.
\end{align*}
The estimate is best possible.
\end{thm}

\begin{thm}\label{p1-070}
Let $f=h+\overline{g}\in\mathcal{F}_0$ is of the form \eqref{p1-001} then for $|z|=r<1,$
\begin{align*}
0\leq |g'(z)|\leq\frac{1}{(1-r)^2}.
\end{align*}
The estimate is best possible.
\end{thm}

\section{Proof of Main Results}\label{Proof of Main Results}

\begin{proof}[\textbf{Proof of Theorem \ref{p1-029}}]
Let $f(z)=h(z)+\overline{g(z)}$ be a sense-preserving harmonic mapping with dilatation $\omega(z)=q^2(z)=g'(z)/h'(z)$. Then
the pre-Schwarzian derivative $\mathbb{P}_f$ of $f$ is given by
\begin{align*}
\mathbb{P}_f(z)=  \frac{h''(z)}{h'(z)}+\frac{2q'(z)\overline{q(z)}}{1+|q(z)|^2}
=  P_h(z)+\frac{2q'(z)\overline{q(z)}}{1+|q(z)|^2}.
\end{align*}
For every $z\in\mathbb{D}$, by Schwarz-Pick lemma, it follows that
\begin{align*}
(1-|z|^2)\Big||\mathbb{P}_f(z)|-|P_h(z)|\Big|
&\leq (1-|z|^2)\Big|\mathbb{P}_f(z)-P_h(z)\Big|\\
&= (1-|z|^2) \left|\frac{2q'(z)\overline{q(z)}}{1+|q(z)|^2}\right|\\
&\leq \frac{2|q(z)|(1-|q(z)|^2)}{1+|q(z)|^2}\\
%& \leq \sup_{z\in\mathbb{D}} \frac{2|q(z)|(1-|q(z)|^2)}{1+|q(z)|^2}\\
&\leq \sup_{q\in\mathbb{D}}\frac{2|q|(1-|q|^2)}{1+|q|^2}\\
&= \sup_{0\leq r <1}\frac{2r(1-r^2)}{1+r^2}\\
&= \frac{2r_0(1-r_0^2)}{1+r_0^2}, ~~\text{where}~ r_0=\sqrt{\sqrt{5}-2}.
\end{align*}
This shows that $||\mathbb{P}_f||$ is finite if and only if $||P_h||$ is finite. Moreover, if $||\mathbb{P}_f||<\infty$ then
\begin{align*}
\Big|||\mathbb{P}_f||-||P_h||\Big|
& \leq \sup_{z\in\mathbb{D}} (1-|z|^2)\Big||\mathbb{P}_f(z)|-|P_h(z)|\Big|\leq \frac{2r_0(1-r_0^2)}{1+r_0^2}\approx 0.6005....
\end{align*}

To show that the estimate is sharp, we consider the sense-preserving harmonic mapping $F(z)=h(z)+\overline{g(z)}=z+\overline{z^3/3}$ with dilatation $\omega(z)=q^2(z)=z^2$. Then $P_h(z)=0$ and so $||P_h||=0.$ Moreover,
\begin{align*}
\mathbb{P}_F(z)=P_h+\frac{2|z|}{1+|z|^2}= \frac{2|z|}{1+|z|^2},
\end{align*}
and so
\begin{align*}
||\mathbb{P}_F||=\sup_{z\in\mathbb{D}}\frac{2|z|(1-|z|^2)}{1+|z|^2}=\sup_{0\leq r<1}\frac{2r(1-r^2)}{1+r^2}=\frac{2r_0(1-r_0^2)}{1+r_0^2},
\end{align*}
where $r_0=\sqrt{\sqrt{5}-2}$. Thus
\begin{align*}
\Big|||\mathbb{P}_F||-||P_h||\Big|=\frac{2r_0(1-r_0^2)}{1+r_0^2}.
\end{align*}
\end{proof}

\begin{proof}[\textbf{Proof of Corollary \ref{p1-030}}]
Since $h\in\mathcal{A}$ satisfies  ${\rm Re\,}Q_h(z)>-\frac{1}{2}$ for $z\in\mathbb{D},$ it follows that (see \cite{Ponnusamy-Sahoo-Yanagihara-2014})
\begin{align*}
\left|\frac{zh''(z)}{h'(z)}\right|\leq \frac{3r}{1-r}\quad\text{for}\quad|z|=r<1
\end{align*}
and so
$$||P_h||=\sup_{z\in\mathbb{D}}(1-|z|^2)\left|\frac{h''(z)}{h'(z)}\right|\leq6.$$
Therefore, by Theorem \ref{p1-029} we can easily get
$$||\mathbb{P}_f||\leq ||P_h||+\frac{2r_0(1-r_0^2)}{1+r_0^2}\approx 6.6005\ldots,~~\text{where}~~r_0=\sqrt{\sqrt{5}-2}.$$
\end{proof}

\begin{proof}[\textbf{Proof of Corollary \ref{p1-031}}]
Since $h\in\mathcal{A}$ satisfies  ${\rm Re\,}Q_h(z)<\frac{3}{2}$ for $z\in\mathbb{D},$ it follows that (see \cite{Maharan-Prajapat-Srivastava-2017})
\begin{align*}
\left|\frac{zh''(z)}{h'(z)}\right|\leq \frac{r}{1-r}\quad\text{for}\quad|z|=r<1
\end{align*}
and so
$$||P_h||=\sup_{z\in\mathbb{D}}(1-|z|^2)\left|\frac{h''(z)}{h'(z)}\right|\leq2.$$
Therefore, by Theorem \ref{p1-029}, one can easily get
$$||\mathbb{P}_f||\leq ||P_h||+\frac{2r_0(1-r_0^2)}{1+r_0^2}\approx 2.6005\ldots,~~\text{where}~~r_0=\sqrt{\sqrt{5}-2}.$$
\end{proof}

\begin{proof}[\textbf{Proof of Corollary \ref{p1-033}}]
Since $h\in\mathcal{A}$ satisfies $-\frac{1}{2}<{\rm Re\,}Q_h(z)<\frac{3}{2}$ for $z\in\mathbb{D}$ i.e. the analytic part $h$ satisfies the hypothesis of the Corollary \ref{p1-030} and Corollary \ref{p1-031} simultaneously and so the pre-Schwarzian norm of $f$ is
\begin{align*}
||\mathbb{P}_f|| & \leq \min\left\{2+\frac{2r_0(1-r_0^2)}{1+r_0^2}, 6+\frac{2r_0(1-r_0^2)}{1+r_0^2}\right\}\\
& = 2+\frac{2r_0(1-r_0^2)}{1+r_0^2}\approx 2.6005\ldots,~~\text{where}~~r_0=\sqrt{\sqrt{5}-2}.
\end{align*}
\end{proof}

Before we prove our next result, let us recall an important and useful tool known as the differential subordination technique. Many problems in geometric function theory can be solved in a simple and sharp manner with the help of differential subordination. Let $g$ and  $h$ be analytic functions in the unit disk $\mathbb{D}$. A function $g$ is said to be subordinate to $h$, written  as
$g\prec h$ or $g(z)\prec h(z)$ if there exists an analytic function $\epsilon: \mathbb{D} \rightarrow \mathbb{D}$ with $\epsilon(0)=0$ such that $g(z)=h(\epsilon(z))$.
If $h$ is univalent, then $g\prec h$ if and only if $g(0)=h(0)$ and $g(\mathbb{D})\subseteq h(\mathbb{D})$. A function $g$ is said to be majorized by $h$ if $|g(z)|\leq|h(z)|$ in the unit disk $\mathbb{D}.$ Equivalently, a function $g$ is said to be majorized by $h$ if there exists an analytic function $\epsilon: \mathbb{D} \rightarrow \mathbb{D}$ such that $g(z)=\epsilon(z)h(z)$. For more details about subordination and memorization, we refer to \cite{Hallenbeck-MacGregor-1984,Pommerenke-1975}.

%%\begin{thm}\label{p1-057}
%%If $f=h+\overline{g}\in \mathcal{F}_0$ is of the form \eqref{p1-001} then $||P_f||\leq5$. The estimate is best possible.
%%\end{thm}

\begin{proof}[\textbf{Proof of Theorem \ref{p1-057}}]
Let $f=h+\overline{g}\in\mathcal{F}_0$. Then $h$ satisfies the subordination relation
\begin{align*}
 1+z\frac{h''(z)}{h'(z)}\prec \frac{1+z}{1-z},
\end{align*}
which gives
\begin{align}\label{p1-075}
\left|\frac{h''(z)}{h'(z)}\right|\leq \frac{2}{1-|z|}.
\end{align}
Let $\omega=g'/h'\in Aut(\mathbb{D})$ be the dilatation of $f=h+\overline{g}.$ Then by Schwarz-Pick lemma  and \eqref{p1-075}, we have
\begin{align}\label{p1-080}
||P_f||= & \sup_{z\in\mathbb{D}}(1-|z|^2)|P_f(z)|\\
      = &  \sup_{z\in\mathbb{D}}(1-|z|^2)\left|\frac{h''(z)}{h'(z)}-\frac{\overline{\omega(z)}\omega'(z)}{1-|\omega(z)|^2}\right|\nonumber\\
       \leq & \sup_{z\in\mathbb{D}}(1-|z|^2)\left(\left|\frac{h''(z)}{h'(z)}\right|+\left|\frac{\overline{\omega(z)}\omega'(z)}{1-|\omega(z)|^2}\right|\right)\nonumber\\
      \leq & \sup_{z\in\mathbb{D}}(1-|z|^2)\left(\frac{2}{1-|z|}+\frac{|\overline{\omega(z)}|}{1-|z|^2}\right)\nonumber\\
      \leq & \sup_{z\in\mathbb{D}}(2(1+|z|)+|\omega(z)|)\nonumber\\
        = & 5\nonumber.
\end{align}
We now show that the estimate is best possible. For each $t\in[1/2,1),$ we consider the function $f_t=h+\overline{g_t}\in\mathcal{H}$ with dilatation $\omega_t(z)=\frac{z-t}{1-tz}$ and $h(z)$ be such  that
\begin{align*}
 1+z\frac{h''(z)}{h'(z)}=\frac{1+z}{1-z}.
\end{align*}
%i.e ,
%\begin{align*}
%\frac{h''(z)}{h'(z)}=\frac{2}{1-z}.
%\end{align*}
Then clearly $f_t\in\mathcal{F}_0$ for all $t\in [1/2,1).$  A simple calculation gives
\begin{equation*}\label{p1-085}
\frac{\overline{\omega_t(z)}\omega_t'(z)}{1-|\omega_t(z)|^2}=\frac{\overline{z}-t}{(1-tz)(1-|z|^2)}
\end{equation*}
and so,
\begin{equation}\label{p1-090}
\begin{split}
||P_{f_t}||
&= \sup_{z\in\mathbb{D}}(1-|z|^2)\left|\frac{h''(z)}{h'(z)}-\frac{\overline{\omega_t(z)}\omega_t'(z)}{1-|\omega_t(z)|^2}\right|\\
&= \sup_{z\in\mathbb{D}}(1-|z|^2)\left|\frac{2}{1-z}-\frac{\overline{z}-t}{(1-tz)(1-|z|^2)}\right|.
\end{split}
\end{equation}
Let
\begin{align*}
M_t= & \sup_{z\in[0,1)}(1-|z|^2)\left|\frac{2}{1-z}-\frac{\overline{z}-t}{(1-tz)(1-|z|^2)}\right|\\
 = & \sup\limits_{r\in [0,1)}\left|2(1+r)-\frac{r-t}{1-tr}\right|\\
 =&\sup\limits_{r\in [0,1)}\psi(r),
\end{align*}
where $\psi(r)=2(1+r)-\frac{r-t}{1-tr}>0$. Then
\begin{align*}
\psi'(r)=2-\frac{1-t^2}{(1-tr)^2}~~\quad\text{and}~~\quad \psi''(r)=-\frac{2t(1-t^2)}{(1-tr)^3}<0~~\text{for all}~~r\in [0,1).
\end{align*}
 Now, $\psi'(r)=0$ gives
\begin{align*}
r=r_0:=\frac{1}{t}-\frac{1}{t}\sqrt{\frac{1-t^2}{2}}.
\end{align*}
So the maximum value of $\psi$ is attained at $r_0.$ Hence
\begin{equation}\label{p1-095}
M_t=\psi(r_0)=2+\frac{3}{t}-\frac{4}{t}\sqrt{\frac{1-t^2}{2}}.
\end{equation}
From  \eqref{p1-090} and \eqref{p1-095} it is clear that $M_t\leq||P_{f_t}||\leq5.$
We note that $M_t$ is an increasing function in $t\in[1/2,1)$ and $M_t\rightarrow5$ as $t\rightarrow1$. This shows that estimate \eqref{p1-080} is best possible.
\end{proof}

%%\begin{thm}\label{p1-060}
%%If $f=h+\overline{g}\in \mathcal{F}_0$ is of the form \eqref{p1-001} then $||S_f||\leq3$. The estimate is best possible.
%%\end{thm}

\begin{proof}[\textbf{Proof of Theorem \ref{p1-060}}]
Let $f=h+\overline{g}\in \mathcal{F}_0.$ Since the class $\mathcal{F}_0$ is linearly invariant, it follows that the Koebe transform $L_\phi(f)=H+\overline{G}$ defined by \eqref{p1-040} also belong to $\mathcal{F}_0$ for $\phi\in Aut(\mathbb{D}).$ Thus the dilatation $\omega$ of $L_\phi(f)$ is of the form
$$\omega= e^{i\theta}\frac{z+\alpha_0}{1+\overline{\alpha_0}z}~~\text{for some}~~\theta\in\mathbb{R}, \alpha_0\in\mathbb{D}$$
and the function $H(z)$ is of the form of \eqref{p1-001} satisfying
\begin{align*}
{\rm Re\,}\left(1+\frac{zH''(z)}{H'(z)}\right)>0\quad\text{for}~ z\in\mathbb{D}.
\end{align*}
%Therefore,
%\begin{align*}
%1+\frac{zh''(z)}{h'(z)}\prec \frac{1+z}{1-z}.
%\end{align*}
Then there exists an analytic function $\epsilon:\mathbb{D}\rightarrow\mathbb{D}$ of the form $\epsilon(z)=\sum\limits_{n=1}^\infty c_nz^n$ such that
\begin{equation*}\label{p1-100}
1+\frac{zH''(z)}{H'(z)}= \frac{1+\epsilon(z)}{1-\epsilon(z)},
\end{equation*}
%where the Taylor series expansion of $\epsilon(z)$ is given by
%$$\epsilon(z)=\sum\limits_{n=1}^{\infty}c_nz^n.$$
and so
\begin{equation}\label{p1-105}
P_H(z)=\frac{H''(z)}{H'(z)}=\frac{2\epsilon(z)}{z(1-\epsilon(z))}=2c_1+(2c_2+2c_1^2)z+\ldots
\end{equation}
From \eqref{p1-005} and \eqref{p1-105}, the  pre-Schwarzian and Schwarzian derivatives of $H$ at $z=0$ is given by
\begin{align}\label{p1-110}
P_H(0)=2c_1,
\end{align}
and
\begin{equation}\label{p1-115}
S_H(0)=P_H'(0)-\frac{1}{2}P_H^2(0)=(2c_2+2c_1^2)-\frac{1}{2}(2c_1)^2=2c_2.
\end{equation}
Since the class $\mathcal{F}_0$ is linearly invariant and the Schwarzian derivative is affine invariant, by chain rule for the Schwarzian derivative (see \cite{Hernandez-Martin-2015}), we have
\begin{align*}
(1-|z|^2)^2|S_f|=|S_{f\circ\phi}(0)|=|S_{L_{\phi}(f)}(0)|
\end{align*}
for each $z\in\mathbb{D}$ and $\phi$ is an automorphism of the disk with $\phi(0) = z$. Hence, it is not very difficult to prove that
$$
\sup_{f\in\mathcal{F}_0}||S_f|| =\sup_{f\in\mathcal{F}_0}|S_f(0)| =\sup_{f\in\mathcal{F}_0}|S_{L_{\phi}(f)}(0)|.
$$

Therefore from \eqref{p1-110} and \eqref{p1-115} it follows that
\begin{equation*}
\begin{split}
&|S_{L_{\phi}(f)}(0)|\\
&= \left|S_H(0)+\frac{\overline{\omega(0)}}{1-|\omega(0)|^2}\left(\frac{H''(0)}{H'(0)}\omega'(0)-\omega''(0)\right)
-\frac{3}{2}\left(\frac{\omega'(0)\overline{\omega(0)}}{1-|\omega(0)|^2}\right)^2\right|\\
&= \left|2c_2+\frac{\overline{\alpha}e^{-i\theta}}{1-|\alpha|^2}\left\{2c_1(1-|\alpha|^2)e^{i\theta}+2\overline{\alpha} e^{i\theta}(1-|\alpha|^2)\right\}-\frac{3}{2}\left(\frac{e^{i\theta}(1-|\alpha|^2)\overline{\alpha}e^{-i\theta}}{1-|\alpha|^2}\right)^2\right|\\
&= |2c_2+2\overline{\alpha}c_1+2\overline{\alpha}^2-\frac{3}{2}\overline{\alpha}^2|\\
&\leq  2|\alpha||c_1|+2|c_2|+\frac{|\alpha|^2}{2}\\
&\leq  2|\alpha||c_1|+2(1-|c_1|^2)+\frac{|\alpha|^2}{2}\\
&\leq  2(1+|\alpha||c_1|-|c_1|^2)+\frac{|\alpha|^2}{2}.
\end{split}
\end{equation*}
It is a simple exercise to see that the function $\psi(|c_1|):=1+|\alpha||c_1|-|c_1|^2,~~0\leq |c_1|\leq1$ has maximum value at $|c_1|=|\alpha|/2.$ Thus,
\begin{align*}
 ||S_f|| \leq 2  \left(1+\frac{|\alpha|^2}{4}\right)+\frac{|\alpha|^2}{2}
 = 2+|\alpha|^2, \quad |\alpha|\in [0,1),
\end{align*}
and so $||S_f||\leq 3.$\\

To illustrate that the estimate is best possible, let us consider the harmonic function $f_t=h_t+\overline{g_t}\in\mathcal{H}$ with dilatation $\omega_t(z)=(z+t)/(1+tz),~ t\in [0,1)$ and  $h_t(z)$ be such that
\begin{equation}\label{p1-120}
1+\frac{zh_t''(z)}{h_t'(z)}=\frac{1+\epsilon_t(z)}{1-\epsilon_t(z)} ,
\end{equation}
where
\begin{align*}
\epsilon_t(z)=z\frac{z+\frac{t}{2}}{1+\frac{t}{2}z}=\sum_{n+1}^\infty c_nz^n=\frac{t}{2}z+\left(1-\frac{t^2}{4}\right)z^2+ \cdots.
\end{align*}
Clearly $f_t\in\mathcal{F}_0$ for all $t\in[0,1).$\\

Now we calculate $(1-|z|^2)S_{f_t}(z)$ at the origin, which is
%To show $||S_{f_0}||=\sup_{z\in\mathbb{D}}|S_{f_0}(z)|(1-|z|^2)^2=3$, it is sufficient to show $|S_{f_0}(z)|(1-|z|^2)^2=3$ at some point $a\in\mathbb{D}$. We take $a=0$ and find the Schwarzian derivative $S_{f_0}$ at $z=0$, which is
\begin{align*}
S_{f_t}(0)=& S_{h_t}(0)+\frac{\overline{\omega_t(0)}}{1-|\omega_t(0)|^2}\left(\frac{h_t''(0)}{h_t'(0)}\omega_t'(0)-\omega_t''(0)\right)-\frac{3}{2}\left(\frac{\omega_t'(0)\overline{\omega_t(0)}}{1-|\omega_t(0)|^2}\right)^2\\
= & 2c_2+\frac{t}{1-t^2}(2c_1(1-t^2)+2t(1-t^2))-\frac{3}{2}\left(\frac{(1-t^2)t}{1-t^2}\right)^2\\
%= & 2c_2+2tc_1+2t^2-\frac{3}{2}\\
%= & 2(1-\frac{t^2}{4})+2t(\frac{t^2}{2})+\frac{t^2}{2}\\
= & 2+t^2, ~t\in [0,1),
\end{align*}
which tends to $3$ as $t\rightarrow 1$. This ensures  that the estimate is best possible.
\end{proof}

To prove the Theorem \ref{p1-065}, we require the following result of Hallenbeck and MacGregor \cite{Hallenbeck-MacGregor-1974}.
\begin{lem}\label{p1-125}
Suppose that an analytic function $g$ of the form $g(z)=\sum\limits_{n=1}^\infty b_nz^n$ is majorized by $h$ where $h$ is a normalized close-to-convex function in the unit disk $\mathbb{D}$. Then $|b_n|\leq n$ for all $n\ge 1$.
\end{lem}

\begin{proof}[\textbf{Proof of Theorem \ref{p1-065}}]
Let $f=h+\overline{g}\in\mathcal{F}_0$ be of the form (\ref{p1-001}) with dilatation $\omega\in Aut(\mathbb{D})$. Then $g'(z)=\omega(z)h'(z)$ and so
\begin{align}\label{p1-130}
zg'(z)=\omega(z)zh'(z)~~\text{where}~~|\omega(z)|<1.
\end{align}
As $h(z)$ is convex univalent function, it follows that $zh'(z)$ is starlike and hence close-to-convex function. Thus from \eqref{p1-130} we see that $zg'(z)$ is majorized by the close-to-convex function $zh'(z)$ and so by the Lemma \ref{p1-125},
\begin{align*}
|nb_n|\leq n ~~\text{for all}~~ n\geq 1
\end{align*}
which implies that
\begin{align*}
|b_n|\leq1.
\end{align*}

To show that the estimate is best possible, we consider the function $f=h+\overline{g}$ with $h(z)=z/(1-z)$ and dilatation $\omega(z)=(z+\gamma)/(1+\gamma z),~\gamma\in[0,1)$. A direct calculation from $g'(z)=h'(z)\omega(z)$ gives
\begin{align*}
g'(z)=\frac{z+\gamma}{(1+\gamma z)(1-z)^2} =\sum\limits_{n=1}^\infty n\left[1-\frac{1}{n}\left(\frac{1-\gamma}{1+\gamma}\right)\left(1-(-\gamma)^n\right)\right]z^{n-1}.
\end{align*}
%\begin{align*}
%g(z) &=\frac{z}{1-z}+\frac{1-\gamma}{1+\gamma}\log{\frac{1-z}{1+\gamma z}}=\sum\limits_{n=1}^\infty b_nz^n\\
%&=\sum\limits_{n=1}^\infty\left[1-\frac{1}{n}\left(\frac{1-\gamma}{1+\gamma}\right)\left(1-(-\gamma)^n\right)\right]z^n.
%\end{align*}
Clearly, the coefficients $\displaystyle b_n=1-\frac{1}{n}\left(\frac{1-\gamma}{1+\gamma}\right)\left(1-(-\gamma)^n\right)$ tend to $1$ as $\gamma\to 1^{-}$ for each $n\ge 1$.
\end{proof}

%\begin{thm}\label{p1-070}
%Let $f=h+\overline{g}\in\mathcal{F}_0$ is of the form \eqref{p1-001} then for $|z|=r<1,$
%\begin{align*}
%\frac{r}{(1+r)^2}\leq |g'(z)|\leq\frac{1}{(1-r)^2}.
%\end{align*}
%The estimate is best possible.
%\end{thm}

\begin{proof}[\textbf{Proof of Theorem \ref{p1-070}}]
Let $f=h+\overline{g}\in\mathcal{F}_0$. Then for $|z|=r<1$ we have \cite{Pinchuk-1968}
\begin{align}\label{p1-140}
\frac{1}{(1+r)^2}\leq|h'(z)|\leq\frac{1}{(1-r)^2}.
\end{align}
As the dilatation $\omega$ of $f$ is of the form $\omega(z)=e^{i\theta}(z+\alpha)/(1+\overline{\alpha}z),$ we have the well known result, (see \cite{Kanas-Maharana-Prajapat-2019})
\begin{align}\label{p1-145}
\frac{|r-\gamma|}{1-\gamma r}\leq|\omega(z)|\leq\frac{r+\gamma}{1+\gamma r}~~\text{where}~~\gamma=|\alpha|.
\end{align}
Using $g'(z)=\omega(z)h'(z)$, \eqref{p1-140} and \eqref{p1-145} for $|z|=r<1,$ we get
\begin{align*}
\frac{|-r+\gamma|}{(1-\gamma r)(1+r)^2}\leq |g'(z)|\leq\frac{r+\gamma}{(1+\gamma r)(1-r)^2} .
\end{align*}
It is easy to verify that the function $|(\gamma-r)|/\{(1-\gamma r)(1+r)^2\}$ attains minimum at $\gamma=r$ and the function $(r+\gamma)/{(1+\gamma r)(1-r)^2}$ attains maximum as $\gamma\rightarrow1.$

\noindent
Thus,
\begin{align*}
0\leq |g'(z)|\leq\frac{1}{(1-r)^2}.
\end{align*}
For the harmonic function $f=h+\overline{g}$ with $h(z)=z/(1-z)$ and dilatation $\omega(z)=(z+\gamma)/(1+\gamma z),~\gamma\in[0,1)$, it is easy to see that
\begin{align*}
g'(z)=\frac{z+\gamma}{(1+\gamma z)(1-z)^2}.
\end{align*}
For $|z|=r<1,$
\begin{align*}
g'(-r)=\frac{-r+\gamma}{(1-\gamma r)(1+r)^2}~\quad\text{and}\quad g'(r)=\frac{r+\gamma}{(1+\gamma r)(1-r)^2}.
\end{align*}
At $\gamma=r$ we see that $g'(-r)=0$ and $g'(r)\to 1/(1-r)^2$ as $\gamma\to 1^{-}$ which shows that the estimate is best possible.
\end{proof}

\section{Concluding Remarks}\label{Concluding Remarks}
In this section, we discuss about some earlier flawed works on the pre-Schwarzian norm $||\mathbb{P}_f||$ and how our result (Theorem \ref{p1-029}) rectify these flawed results.
In 2014, Kanas and Klimek-Sm\c{e}t \cite{Kanas-Smet-2014} considered two subclasses $\mathcal{S}^{\alpha}_{\mathcal{H}}$ and $\mathcal{V}^{\mathcal{H}}(k)$ of $\mathcal{H}$ of harmonic mappings $f = h+\overline{g}$ such that $|g'(0)| = \alpha\in [0,1)$ and $g'(z) = \omega(z)h'(z)$ where $h \in \mathcal{A}$ satisfies certain analytic criteria and obtain estimate of the pre-Scwarzian norm $||\mathbb{P}_f||$ in terms of $\alpha$.
Here we would like to point out that the expression for the Schwarzian derivative $\mathbb{S}_f$ and pre-Schwarzian derivative $\mathbb{P}_f$ in \cite[Eqns. (3.4) and (3.5)]{Kanas-Smet-2014} are not correct for a harmonic mapping $f = h+\overline{g} \in \mathcal{H}$ with the dilatation $\omega=g'/h'$, but these expression are correct when the dilatation is $\omega^2=g'/h'$ where $\omega:\mathbb{D}\to \mathbb{D}$ is an analytic function. Consequently, if $\omega^2=g'/h'$ is the dilatation of $f = h+\overline{g}\in\mathcal{H}$ with $|g'(0)| = \alpha\in [0,1)$ then $\omega(0)=\sqrt{\alpha}$. Due to this normalization, the estimate of the pre-Schwarzian norm $||\mathbb{P}_f||$ for the classes $\mathcal{S}^{\alpha}_{\mathcal{H}}$ and $\mathcal{V}^{\mathcal{H}}(k)$ obtained in \cite[Theorem 3.3 and 3.4]{Kanas-Smet-2014} are not correct.\\

In 2019, Kanas et al. \cite{Kanas-Maharana-Prajapat-2019} introduced the class $\mathcal{G}_{Har}^\alpha$ for $\alpha\in[0,1)$, consisting of all harmonic mappings $f = h+\overline{g} \in \mathcal{H}$, such that $g'(0) = \alpha\in [0,1)$ and $g'(z) = \omega(z)h'(z)$ where the dilatation $\omega$ is a bilinear transformation and $h \in \mathcal{G}$ with
$$\mathcal{G}=\{h\in\mathcal{A}:-1/2<{\rm Re\,}Q_h(z)<3/2 ~\text{for}~ z\in\mathbb{D}\}.$$
Kanas et al. \cite{Kanas-Maharana-Prajapat-2019} obtained estimate of the pre-Scwarzian norm $||\mathbb{P}_f||$ in terms of $\alpha$ for the class $\mathcal{G}_{Har}^\alpha$. But due to the normalization mistake in dilatation $\omega$ as mentioned above, the estimate of $||\mathbb{P}_f||$ obtained in \cite[Theorem 2.1]{Kanas-Maharana-Prajapat-2019} is not correct.\\
%We also note that Lemmas 1.1 and 1.2 in \cite{Kanas-Maharana-Prajapat-2019} may not be true for the class $\mathcal{G}$, as these results holds only for functions $h \in \mathcal{A}$ with ${\rm Re\,}Q_h(z)<\frac{3}{2}$ for $z\in\mathbb{D}$ (see \cite[Proposition 1]{Maharan-Prajapat-Srivastava-2017}, \cite[Theorem 1, Corollary 2]{Obradovic-Ponnusamy-Wirths-2013}). As the Lemma 1.1 of \cite{Kanas-Maharana-Prajapat-2019} has been used to obtain the estimate of $||\mathbb{P}_f||$ for the class $\mathcal{G}_{Har}^\alpha$, the estimate is not correct. Moreover, in the proof of \cite[Theorem 2.1]{Kanas-Maharana-Prajapat-2019}, the expression of $G(r)$ is not accurate which also implies that the estimate is not correct.\\
%Further, Lemmas 1.1 and 1.2 of \cite{Kanas-Maharana-Prajapat-2019} have been used for the proofs of all the theorems in \cite{Kanas-Maharana-Prajapat-2019} and so these results are not correct.

In 2020, Prajapat et al. \cite{Prajapat-Manivannan-Maharana-2020} studied the class $\mathcal{F_H}(\alpha)$ for $0\le \alpha<1$, consisting of all harmonic functions $f = h +\overline{g}\in \mathcal{H}$ such that $|g'(0)| = \alpha$ and $h \in \mathcal{A}$ satisfies ${\rm Re\,}Q_h(z)>-\frac{1}{2}$ for $z\in\mathbb{D}$. In 2021, Rajbala and Prajapat \cite{Rajbala-Prajapat-2021} introduced the class $G_{\mathcal{H}}(\alpha, \beta)$ for $ 0 \leq \alpha < 1$ and $2/3 < \beta \leq 1$, consisting of all harmonic mappings $f = h+\overline{g} \in \mathcal{H}$ such that $|g'(0)|=\alpha$ and $h \in \mathcal{A}$ satisfies ${\rm Re\,}Q_h(z)<\frac{3}{2}\beta$ for $z\in\mathbb{D}$.
Estimate of the pre-Schwarzian norms $||\mathbb{P}_f||$ and $||P_f||$ for the classes $\mathcal{F_H}(\alpha)$ and $G_{\mathcal{H}}(\alpha,\beta)$ has been obtained by Prajapat et al. \cite{Prajapat-Manivannan-Maharana-2020} and Rajbala and Prajapat \cite{Rajbala-Prajapat-2021} respectively.
Although, the flaw in the normalization of the dilatation has been rectified in \cite[Theorem 1]{Prajapat-Manivannan-Maharana-2020} and \cite[Theorem 2]{Rajbala-Prajapat-2021}, the estimate of $||\mathbb{P}_f||$ for the classes $\mathcal{F_H}(\alpha)$ and $G_{\mathcal{H}}(\alpha,\beta)$ are not sharp, because the norm $||\mathbb{P}_f||$ of the extremal functions (as claimed) obtained in \cite[Theorem 1]{Prajapat-Manivannan-Maharana-2020} and \cite[Theorem 2]{Rajbala-Prajapat-2021} are not correct.\\

As rectifying all these flawed results separately under the different assumptions on the analytic part $h$ is a tedious job, we proved the general result Theorem \ref{p1-029}. In particular, we note that every function $f=h+\overline{g}$ in the class $\mathcal{F_H}(\alpha)$, $0\le \alpha<1$ with dilatation $\omega(z)=q^2(z)=g'(z)/h'(z)$ satisfy the hypothesis of Corollary \ref{p1-030}. Therefore the estimate of the pre-Schwarzian norm $||\mathbb{P}_f||$ obtained in Corollary \ref{p1-030} also holds for functions in $\mathcal{F_H}(\alpha).$ Moreover, this estimate is better than the estimate obtained in \cite{Prajapat-Manivannan-Maharana-2020}. Again, we note that every function $f=h+\overline{g}$ in the class  $G_{\mathcal{H}}(\alpha, \beta)$, $0 \leq \alpha < 1,~2/3 < \beta \leq 1$ with dilatation $\omega(z)=q^2(z)=g'(z)/h'(z)$ satisfy the hypothesis of Corollary \ref{p1-031}. Therefore the estimate of the pre-Schwarzian norm $||\mathbb{P}_f||$ obtained in Corollary \ref{p1-031} also holds for functions in $G_{\mathcal{H}}(\alpha, \beta).$ Moreover, this estimate is better than the estimate obtained in \cite{Rajbala-Prajapat-2021}. Next, we note that every function $f$ in the class $\mathcal{G}_{Har}^\alpha$ with dilatation $\omega(z)=q^2(z)=g'(z)/h'(z)$ satisfy the hypothesis of Corollary \ref{p1-033}. Therefore the estimate of pre-Schwarzian norm $||\mathbb{P}_f||$ obtained in Corollary \ref{p1-033} also holds for such functions.\\

\noindent\textbf{Data availability:}
Data sharing not applicable to this article as no data sets were generated or analyzed during the current study.\\

\noindent\textbf{Authors Contributions:}
All authors contributed equally to the investigation of the problem and the order of the authors is given alphabetically according to their surname. All authors read and approved the final manuscript. \\

\noindent\textbf{Acknowledgement:}
The second named author thanks the Department Of Science and Technology, Ministry Of Science and Technology, Government Of India
for the financial support through DST-INSPIRE Fellowship (No. DST/INSPIRE Fellowship/2018/IF180967).

\end{document}